\documentclass[12pt]{amsart}
\usepackage{amsmath, amsthm, amssymb, fullpage, enumerate, amsfonts, xypic, hyperref}

\def\ZZ{{\mathbb Z}}

\def\FF{{\mathbb F}}
\def\Fbar{{\overline\FF}}

\def\QQ{{\mathbb Q}}

\def\MM{{\mathrm M}}

\def\OO{{\mathcal O}}

\def\pp{{\mathfrak p}}

\def\aa{{\mathfrak a}}

\def\mm{{\mathfrak m}}
\def\ff{{\mathfrak f}}

\def\Pic{{\text{\textnormal{Pic}}}}
\def\Ann{{\text{\textnormal{Ann}}}}
\def\Hom{{\text{\textnormal{Hom}}}}
\def\Ext{{\text{\textnormal{Ext}}}}
\def\End{{\text{\textnormal{End}}}}

\newtheorem{theorem}{Theorem}[section]

\newtheorem{lemma}[theorem]{Lemma}
\newtheorem{prop}[theorem]{Proposition}

\newenvironment{introtheorem2}[1]{\vspace{2.5mm}\noindent\textbf{#1.}\em }{\vspace{2.5mm}}

\usepackage[dvipsnames]{xcolor}

\title{The Structure of the Group of Rational Points of {an~Abelian Variety} over a Finite Field}
\author{Caleb Springer}
\address{Department of Mathematics, The Pennsylvania State University, University Park, PA 16802, USA}
\email{cks5320@psu.edu}
\date{\today}

\begin{document}
\maketitle

\begin{abstract}
    Let $A$ be a simple abelian variety of dimension $g$ defined over a finite field $\FF_q$ with Frobenius endomorphism $\pi$.  This paper describes the structure of the group of rational points $A(\FF_{q^n})$, for all $n \geq 1$, as a module over the ring $R$ of endomorphisms which are defined over $\FF_q$, under certain technical conditions. If $[\QQ(\pi) : \QQ] = 2g$ and $R$ is a Gorenstein ring, then ${A(\FF_{q^n}) \cong R/R(\pi^n-1)}$.  This includes the case when $A$ is ordinary and has maximal real multiplication. Otherwise, if $Z$ is the center of $R$ and $(\pi^n - 1)Z$ is the product of invertible prime ideals in $Z$, then $A(\FF_{q^n})^d \cong R/R(\pi^n - 1)$ where $d = 2g/[\QQ(\pi):\QQ]$.  Finally, we deduce the structure of $A(\Fbar_q)$ as a module over $R$ under similar conditions. These results generalize results of Lenstra for elliptic curves.
\end{abstract}

\section{Introduction}

Given an abelian variety $A$ over a finite field $\FF_q$, one may view the group of rational points $A(\FF_{q})$ as a module over the ring $\End_{\FF_q}(A)$ of endomorphisms defined over $\FF_q$.  Lenstra completely described this module structure for  elliptic curves over finite fields in the following theorem. In addition to being useful and interesting in its own right, this theorem also determines \emph{a fortiori} the underlying abelian group structure of $A(\FF_q)$ purely in terms of the endomorphism ring.  The latter perspective has been leveraged for the sake of computational number theory and cryptography; see, for example, the work of Galbraith \cite[Lemma 1]{galbraith}, Ionica and Joux \cite[\S2.3]{ij}, and Kohel \cite[Chapter 4]{kohel}.  The goal of this paper is to generalize Lenstra's theorem beyond elliptic curves to abelian varieties of arbitrary dimension.

\begin{theorem}[{\cite{lenstra}, Theorem 1}]
\label{len-thm}
	Let $E$ be an elliptic curve over $\FF_q$. Write $R = \End_{\FF_q}(E)$ and let $\pi \in R$ be the Frobenius endomorphism of $E$.  
	\begin{enumerate}[(a)]
	\item  Suppose that $\pi \notin\ZZ$.  Then $R$ has rank $2$ over $\ZZ$ and there is an isomorphism of $R$-modules
	$$
		E(\FF_{q^n}) \cong R/(\pi^n - 1)R.
	$$
	\item Suppose that $\pi \in \ZZ$.  Then $R$ has rank $4$ over $\ZZ$, we have 
	$$
		E(\FF_{q^n}) \cong \ZZ/\ZZ(\pi^n - 1) \oplus \ZZ/\ZZ(\pi^n - 1)
	$$
	as abelian groups.
	 Further, this group has up to isomorphism exactly one left $R$-module structure, and one has an isomorphism of $R$-modules
	$$
		E(\FF_{q^n}) \oplus E(\FF_{q^n}) \cong R/ R(\pi^n- 1).
	$$
	\end{enumerate}
\end{theorem}

Notice that $E$ is supersingular in the second case, but not conversely.  To prove the theorem, Lenstra notes that $E(\FF_{q^n}) = E[\pi^n - 1]$, and $\pi^n - 1$ is a separable isogeny.  For part (b), the abelian group structure is simply the well-known structure of the $n$-torsion of an elliptic curve for $n\in \ZZ$.  The additional statements in part (b) follow from Morita equivalence and an isomorphism of rings,  for integers $n$ coprime to $q$, between $R/Rn$ and the ring $\MM_2( \ZZ/n\ZZ)$ of $2\times 2$ matrices with coefficients in $\ZZ/n\ZZ$.
 
  For part (a) of the theorem, Lenstra uses the following proposition; see \cite[{Proposition~2.1}]{lenstra}.
\begin{prop}
\label{len-key-prop}
Let $E$ be an elliptic curve over $\FF_q$, and let $R = \End_{\FF_q} E$.  If $[R: \ZZ] = 2$, then for every separable element $s\in R$ there is an isomorphism $E[s] \cong R/Rs$ of $R$-modules.
 \end{prop}

 Lenstra  showed in his original paper that the preceding proposition does not immediately generalize to all ``nice" abelian varieties of higher dimension, i.e. principally polarized ordinary abelian varieties;  see \cite[Proposition 6.4]{lenstra}.  Although this means that a certain natural generalization is not correct, the examples that Lenstra produces must have very particular endomorphism rings. By inspecting Lenstra's theorem through two perspectives and imposing restrictions on the endomorphism ring, we can recover a  natural generalization to certain abelian varieties of higher dimension.

 \subsection{First Perspective: Gorenstein Rings}
First, consider part (a) of Lenstra's theorem, or more generally, Proposition \ref{len-key-prop}. In this case, the endomorphism ring of the elliptic curve is commutative, specifically an order in an imaginary quadratic number field.  In general, a simple abelian variety $A$ of dimension $g$ over $\FF_q$ with Frobenius endomorphism $\pi$ has commutative endomorphism ring exactly when $[\QQ(\pi) : \QQ] = 2g$, and in this case, $\End_{\FF_q}(A)$ is an order in the field $\QQ(\pi)$ \cite[Theorem 8]{wm}.  In fact, if $\pi$ is an ordinary Weil $q$-integer, then the rings which arise as the endomorphism rings of abelian varieties in the corresponding isogeny class over $\FF_q$ are precisely the orders of $\QQ(\pi)$ which contain the minimal order $\ZZ[\pi, \overline\pi]$ \cite[Theorem 7.4]{waterhouse}.  Since every order in a quadratic number field is Gorenstein, restricting to the Gorenstein case for abelian varieties of arbitrary dimension provides us with our first natural generalization.

 \begin{introtheorem2}{Proposition~\ref{gor-key-prop}}
	Let $A$ be a simple abelian variety over $\FF_q$ of dimension $g$ with Frobenius endomorphism $\pi$.  If $[\QQ(\pi) : \QQ] = 2g$ and $R = \End_{\FF_q}(A)$ is a Gorenstein ring, then there is an isomorphism of $R$-modules
	$$
		A[s] \cong R/Rs
	$$
	 for every separable $s\in R$.
\end{introtheorem2}

This proposition will be proved in Section \ref{sec:gor} by using properties of finite local Gorenstein rings. 
To see examples where the proposition applies, 
note that
$\End_{\FF_q}(A)$ is guaranteed to be Gorenstein if $A$ has \emph{maximal real multiplication}, i.e. if $\End_{\FF_q}(A)$ contains the ring of integers of the maximal totally real subfield of $\QQ(\pi)$; see \cite[Lemma 4.4]{bjw}.  Many recent results in the algorithmic study of abelian varieties over finite fields have productively focused on the case of maximal real multiplication, including results on point counting \cite{bglgmmst, gks}, isogeny graphs \cite{bjw, it, martindale}, and endomorphism ring computation \cite{springer}.
At the other extreme, Centeleghe and Stix have shown that the minimal order $\ZZ[\pi, \overline\pi]$ is also always Gorenstein, where $\pi$ is a Weil integer \cite[Theorem 11]{cs}.

\subsection{Second perspective: Modules over the center}
Now consider part (b) of Lenstra's theorem, where $E$ is a supersingular elliptic curve over $\FF_q$ with all endomorphisms defined.  Before describing the group of rational points $E(\FF_{q^n})$ as a module over the endomorphism ring $\End_{\FF_q}(E)$, Lenstra first identifies $E(\FF_{q^n})$ as an abelian group, i.e. a module over $\ZZ$.  Importantly, $\ZZ$ is the center of the endomorphism ring in this case.

Following this point of view, given a simple abelian variety $A$ over $\FF_q$ with Frobenius endomorphism $\pi$, we will first consider the structure of $A(\FF_{q^n})$ as a module of the center of $\End_{\FF_q}(A)$. Recall that the center of the endomorphism algebra $\End_{\FF_q}(A)\otimes \QQ$ is the field $\QQ(\pi)$ \cite[Theorem 8]{wm}. More generally, we can study $A[s]$ as a module over the center of the endomorphism ring $\End_{\FF_q}(A)$ for any separable endomorphism $s$ in the center, which leads us to the following result.

\begin{introtheorem2}{Proposition~\ref{inv-key-prop}}
Let $A$ be a simple abelian variety over $\FF_q$ of dimension $g$, and let $Z$ be the center of  $R = \End_{\FF_q}(A)$.  If $s$ is a separable element of $Z$ for which $sZ$ is the product of invertible prime ideals in $Z$, then there is an isomorphism of $Z$-modules
	$$
		A[s] \cong (Z/Zs)^d
	$$
where $d = 2g/[\QQ(\pi):\QQ]$.
Moreover, this $Z$-module has exactly one $R$-module structure, up to isomorphism.
The unique $R$-module structure comes from the isomorphism of rings ${R/Rs \cong \MM_d(Z/Zs)}$, and there is an isomorphism
$$
	A[s]^{d} \cong R/Rs
$$
as $R$-modules.
 \end{introtheorem2}
 
 This proposition will be proved in Section \ref{sec:ker} through the study of kernel ideals.
  The latter parts of this proposition will follow from Morita equivalence, similarly to Theorem \ref{len-thm}.(b). 
Notice that we must require that  $sZ$ is the product of invertible prime ideals, which is automatically true when $Z$ is a maximal order.
For example, let $A$ be an abelian surface defined over $\FF_p$ in the isogeny class corresponding to the Weil polynomial $(t^2 - p)^2$ for a prime $p\not\equiv1\bmod4$.  This Weil polynomial corresponds to the Weil restriction of a supersingular elliptic curve over $\FF_{p^2}$, and $A$ is simple over $\FF_p$.  The endomorphism ring $\End_{\FF_p}(A)$ is a noncommutative ring whose center is $\ZZ[\sqrt p]$, which is a maximal order by construction because $p\not\equiv 1\bmod4$. Hence the proposition automatically applies in this case for any separable $s\in \ZZ[\sqrt p]$.

 \subsection{Main Result}
 
 Combining the perspectives outlined above, we have the following main result.
 \begin{theorem}
 \label{thm:main-result}
 For  $g\geq 1$, let $A$ be a simple abelian variety over $\FF_q$ of dimension $g$ with Frobenius endomorphism $\pi$.  Write $K = \QQ(\pi)$ and $R = \End_{\FF_q}(A)$, and let $Z$ be the center of $R$. 
 
 \begin{enumerate}[(a)]
 \item  If $[K : \QQ] = 2g$ and $R$ is a Gorenstein ring, then
 $$
 	A(\FF_{q^n}) \cong R/R(\pi^n - 1).
 $$
 \item If $(\pi^n - 1)Z$ is the product of invertible prime ideals in $Z$, then there is an isomorphism of $Z$-modules
 $$
		A(\FF_{q^n}) \cong (Z/Z(\pi^n - 1))^{d},
$$
where $d = 2g/[K:\QQ]$.
Moreover, this $Z$-module has exactly one left $R$-module structure, up to isomorphism.
This $R$-module structure comes from the isomorphism of rings ${R/R(\pi^n-1) \cong \MM_d(Z/Z(\pi^n - 1))}$, and there is an isomorphism of $R$-modules
$$
	A(\FF_{q^n})^{ d} \cong R/R(\pi^n - 1).
$$
 \end{enumerate} 
 \end{theorem}
 
   Notice that parts (a) and (b) of the theorem provide the same answer in the case when all hypotheses are simultaneously satisfied, e.g. when $A$ is a simple ordinary abelian variety with maximal endomorphism ring.    The theorem follows immediately from the propositions above, given that ${A(\FF_{q^n}) = A[\pi^n - 1]}$ and $\pi^n - 1$ is a separable isogeny, as in the elliptic curve case.  Propositions \ref{gor-key-prop} and \ref{inv-key-prop} will be proved in Sections \ref{sec:gor} and \ref{sec:ker}, respectively, which completes the proof of our main theorem.
   Finally, in Section \ref{sec:Fbar}, we stitch together all of the isomorphisms described above to understand the structure of $A(\Fbar_q)$ as a module of the endomorphism ring $\End_{\FF_q}(A)$.

 \subsection{Acknowledgements}
 The author thanks Kirsten Eisentr\"ager and Stefano Marseglia for their helpful comments, and thanks Yuri Zarhin for suggesting a simplified approach to Lemma \ref{lem:submod}.
 The author was partially supported by National Science Foundation award CNS-1617802.
 
 %
 %
\section{Gorenstein Rings}
\label{sec:gor}

The goal of this section is to prove the following generalization of Proposition \ref{len-key-prop}, as outlined in the introduction.

 \begin{prop}
 {\label{gor-key-prop}}
	Let $A$ be a simple abelian variety over $\FF_q$ of dimension $g$ with Frobenius endomorphism $\pi$.  If $[\QQ(\pi) : \QQ] = 2g$ and $R = \End_{\FF_q}(A)$ is a Gorenstein ring, then there is an isomorphism of $R$-modules
	$$
		A[s] \cong R/Rs
	$$
	 for every separable $s\in R$.
\end{prop}

In order to prove this proposition, we will follow a strategy that is largely similar to the proof of Theorem \ref{len-thm}.(a) in Lenstra's original paper.  Our approach differs from Lenstra by  working directly with finite local Gorenstein rings, rather than using duality.  Background for Gorenstein rings can be found in Matsumura's book \cite[Chapter 18]{matsumura}.

\begin{lemma} 
\label{lem:submod}
Let $R$ be a Gorenstein domain and $s$ a nonzero element of $R$.  If the quotient $S = R/Rs$ is finite, then every faithful $S$-module $M$ contains a submodule that is free of rank 1 over $S$.
 \end{lemma}
  \begin{proof} 
 Notice that $S$ is Gorenstein because $R$ is Gorenstein; see \cite[Exercise 18.1]{matsumura}.  Additionally, the fact that $S$ is finite implies that it is an Artinian ring. In particular, it is canonically isomorphic to a finite product of its localizations $S = S_1\times \dots \times S_r$.   Thus every $S$-module $M$ has the form $M \cong M_1 \times \dots \times M_r$ where $M_i$ is an $S_i$-module for each $1\leq i\leq r$.  This lemma therefore reduces to the following lemma.
  \end{proof}
  
  \begin{lemma} 
  Let $(T, \mm)$ be a finite local Artinian ring that is Gorenstein.
  \begin{enumerate}[(a)]
  \item
	  Every nonzero  ideal $J\subseteq T$ contained in $\mm$ contains a nonzero element that is killed by all elements of $\mm$.
\item 
	Every faithful $T$-module $N$ contains a submodule that is free of rank 1 over $T$.
  \end{enumerate}
   \end{lemma}
    \begin{proof} 
To prove part (a), list the elements of the maximal ideal $\mm = \{a_1,\dots, a_d\}$.  Define $J_0 = J$, and for each $1\leq i\leq d$, let $J_i$ be the set of elements of $J$ which are annihilated by $\{a_1,\dots, a_i\}$.  
In other words, for each $1\leq i \leq d$, the ideal $J_i$ is the kernel of the map $f_i: J_{i - 1}\to J_{i - 1}$ defined by $x\mapsto a_ix$.
All elements of $\mm$ are nilpotent, and therefore the kernel $J_i$ of the map $f_i$ is nontrivial precisely when $J_{i -1}\neq 0$.   Since $J_0 \neq 0$ by hypothesis, it is clear by induction that $J_i \neq 0$ for all $1\leq i\leq d$.  In particular, there are nonzero elements in $J_d\subseteq J$ which are annihilated by every element of $\mm$.
  
   For part (b), let $k = T/\mm$ be the residue field of $T$.  Because $T$ is a zero-dimensional Gorenstein ring, the $k$-vector space $\Ext_T^0(k,T) = \Hom_T(k,T)$ is one-dimensional; see \cite[Theorem~18.1]{matsumura}.  Thus the annihilator of $\mm$ in $T$ is a principal ideal $I = t T$ where $t =\phi(1)$ for some nonzero $\phi: k\to T$.  Because $N$ is a faithful module, there is some $n\in N$ such that $tn\neq 0$. Let $\Ann(n)$ be the annihilator of $n$, which is an ideal contained in $\mm$.  
   
   If $\Ann(n) = 0$, then the submodule $T n \subseteq N$ is free of rank 1 and we are done. If $\Ann(n)\neq 0$, then part (a) implies that $\Ann(n)$ contains a nonzero element $x$ which is killed by all elements of $\mm$.  Since $I$ is the annihilator of $\mm$, this means that $x\in \Ann(n)$ is also a nonzero element of $I$. However, $I$ is a principal ideal that can be viewed as a module over the field $k = T/\mm$, hence every nonzero element of $I$ is a generator.  
In particular, $x n \neq 0$ because $t\in I = xT$ and  $tn\neq 0$. This contradiction completes the proof.
    \end{proof}

 We are now ready to prove the key proposition.
 \begin{proof}[Proof of Proposition {\ref{gor-key-prop}}]
	Put $S = R/Rs$ and $M = A[s]$ for ease of notation.  
	Notice that $M$ is a \emph{faithful} $S$-module: Any $r\in R$ such that $rM = rA[s] = 0$ factors as $r = ts$ for some $t\in R$, i.e. $r\in Rs$.  Indeed, this follows immediately from the universal property of quotients; see \cite[Remark 7.(c)]{kani}.
	
	Therefore, Lemma \ref{lem:submod} implies that $M$ contains a free $S$-submodule of rank 1.  Now, we can count the cardinalities of these sets:
	$$
		\#M = \deg s = N_{K/\QQ}s = \#R/Rs = \# S.
	$$
	The first equality comes from the separability of $s$, and the second equality above is a well-known theorem \cite[Proposition V.12.12]{milne}.  Therefore, $M\cong S$ as an $S$-module because their cardinalities are the same. This proves Proposition \ref{gor-key-prop}.
 \end{proof}

 %
 %
\section{Using Kernel Ideals}
\label{sec:ker}

 In this section, $A$ is a simple abelian variety over $ \FF_q$ with Frobenius endomorphism $\pi$.  Then the endomorphism algebra $D = \End_{\FF_q}(A)\otimes \QQ$ is a division algebra with center $K = \QQ(\pi)$ \cite[Theorem 8]{wm}. Write $R = \End_{\FF_q}(A)$, and let $Z$ be the center of the endomorphism ring.  Our goal in this section is to prove Proposition \ref{inv-key-prop}, which we repeat below for convenience.

\begin{prop} 
\label{inv-key-prop}
 If $s$ is a separable element of $Z$ for which $sZ$ is the product of invertible prime ideals in $Z$, then  there is an isomorphism of $Z$-modules
	$$
		A[s] \cong (Z/Zs)^d
	$$
where $d = 2g/[\QQ(\pi):\QQ]$.
Moreover, this $Z$-module has exactly one $R$-module structure, up to isomorphism.  This $R$-module structure comes from the isomorphism of rings ${R/Rs \cong \MM_d(Z/Zs)}$, and there is an isomorphism
$$
	A[s]^{d} \cong R/Rs
$$
as $R$-modules.
 \end{prop}

To prove this proposition, we will inspect the isogenies associated to (left) ideals, inspired by Waterhouse \cite{waterhouse}; see also \cite[\S2]{kani} for additional background. In the construction of Waterhouse, a nonzero ideal $I\subseteq R$ is associated to an isogeny whose kernel is $A[I] = \cap_{\alpha\in I} A[\alpha]$, where $A[\alpha]$ is the kernel of the endomorphism $\alpha$.  In other words, if $I$ is generated by the elements $\alpha_1,\dots, \alpha_m$, then the abelian variety $A/A[I]$ is isomorphic to the image of the map $(\alpha_1,\dots, \alpha_m): A \to A^m$.

Similarly, we can also associate a finite subgroup scheme $H$ of $A$ to a left ideal $I(H) \subseteq R$, given by
$$
	I(H) = \{\alpha \in R : H \subseteq A[\alpha]\}.
$$
Given a nonzero ideal $I\subseteq R$, we always have $I \subseteq I(A[I])$.  If equality holds, then $I$ is called a \emph{kernel ideal}.  Every nonzero ideal $I$ is contained in a kernel ideal $J$ such that $A[I] = A[J]$.

For our purposes, we will be concerned with isogenies that are associated to ideals contained in the center $I_0\subseteq Z$.  For convenience, we will write $A[I_0]$ in place of $A[I_0R]$.   The goal of this section is to describe $A[s]$ in terms of $A[\pp_j^{e_j}]$ where  $sZ = \pp_1^{e_1}\dots \pp_r^{e_r}$ is the factorization of $s$ into invertible prime ideals in $Z$, which will allow us to prove Proposition \ref{inv-key-prop}.

\subsection{Basics of invertible ideals}
 First, we recall some basic key properties about invertible ideals in algebraic number theory.  Within this section, let $L$ denote a number field and let $\OO\subseteq L$ be an order. The \emph{conductor ideal} of $\OO$ is defined to be $\ff_\OO = \{a\in L : a\OO_L \subseteq \OO \}$. The following  lemmas  show the connection between the conductor ideal and the invertibility of ideals.

\begin{lemma} 
\label{lem:inv_dvr}
If $\pp\subseteq \OO$ is a nonzero prime ideal, then the following are equivalent:
\begin{enumerate}
	\item $\pp$ is invertible, i.e. $\pp I = a\OO$ for some ideal $I\subseteq \OO$ and some $a\in \OO$;
	\item $\pp$ is regular, i.e. the localization $\OO_\pp$ is integrally closed;
	\item $\pp$ is coprime to the conductor ideal $\ff_\OO $, i.e. $\pp + \ff_\OO = \OO$.
\end{enumerate}
Moreover, when these equivalent conditions hold, the localization $\OO_\pp$ is a discrete valuation ring.
 \end{lemma}
  \begin{proof} 
The prime ideal $\pp$ is invertible if and only if it is regular by \cite[Exercise I.12.5]{Neukirch}, which is true if and only if $\pp\not\supseteq \ff_\OO$  \cite[Proposition 12.10]{Neukirch}. To obtain the last equivalent condition, observe that $\OO$ is a one-dimensional Noetherian integral domain \cite[Proposition I.12.2]{Neukirch}, so any nonzero prime ideal of $\OO$ is maximal. In particular, $\pp \not\supseteq \ff_\OO$ is equivalent to $\pp + \ff_\OO = \OO$.  

Finally, if $\pp$ is regular, then the localization  $\OO_\pp$ is equal to the localization of the ring of integers $\OO_L$ at the prime ideal $\hat{\pp} = \pp\OO_L$ \cite[Proposition 12.10]{Neukirch}, and the latter localization $\OO_{L, \hat{\pp}}$ is known to be a discrete valuation ring \cite[Proposition I.11.5]{Neukirch}.
  \end{proof}
  
  While the preceding lemma focuses on prime ideals, the following result shows the connection between invertibility and the conductor ideal in general.  In particular, we see that Proposition \ref{inv-key-prop} can be rephrased to require that $sZ$ is coprime to the conductor ideal $\ff_Z$ of $Z$ instead of requiring that $sZ$ is the product of invertible ideals.

\begin{lemma}[Proposition 3.2,\cite{LD}]
\label{lem:unique-fac}
If $\aa\subseteq \OO$ is any ideal coprime to the conductor $\ff_\OO$, then $\aa$ is invertible and is uniquely factored into (invertible) prime ideals.
 \end{lemma}
 
Recall that the \emph{Picard group} $\Pic(\OO)$ is defined to be the quotient of the set of invertible fractional ideals of $\OO$ by the set of principal fractional ideals.  We refer readers to \cite[\S I.12]{Neukirch} and \cite{LD} for additional background. 

 \begin{lemma} 
Every class of ideals in $\Pic(\OO)$ contains infinitely many prime ideals.
   \label{lem:choose_inverse}
  \end{lemma}
   \begin{proof} 
   The extension and contraction of ideals provides a natural bijection between the set of invertible prime ideals of $\OO$ and the set of prime ideals of $\OO_L$ which are coprime to the conductor ideal $\ff_\OO$ \cite[Lemma 3.3]{LD}.  Using this bijection, there is a natural isomorphism of groups that allows us to interpret the Picard group $\Pic(\OO)$ in terms of fractional ideals of $\OO_L$ which are coprime to the ideal $\ff_\OO$ \cite[Theorem 3.11]{LD}. This reduces the claim to a question concerning ideals in $\OO_L$, and a generalization of the Dirichlet density theorem immediately shows that there are infinitely many suitable prime ideals \cite[Theorem VII.13.2]{Neukirch}.
   \end{proof}

  \subsection{Isogenies associated to ideals}
  Now we focus our attention on the invertible ideals of the center $Z$ of the endomorphism ring $R$, and investigate the corresponding isogenies.

\begin{lemma}
\label{kernel-ideal-lemma}
If $I_0\subseteq Z$ is an invertible ideal, then $I_0R$ is an invertible two-sided ideal of $R$.  In particular, $I_0R$ is a kernel ideal.
 \end{lemma}
  \begin{proof} 
Clearly $I_0R$ is naturally a right ideal, and $RI_0$ is naturally a left ideal, and these two sets are equal as $I_0\subseteq Z$ is in the center. Thus, $I_0R$ is a two-sided ideal.
  
 Because $I_0$ is invertible, there is a fractional ideal $J_0$ of $Z$  such that $I_0J_0 = Z$.  Since $Z$ is the center of $R$, it also follows that 
 $$
 	(I_0R)(J_0R) = (J_0R)(I_0R) = R.
$$  Moreover, if $J$ is any fractional two-sided ideal of $R$ such that $J\cdot(I_0R) = (I_0R)\cdot J = R$, then $J_0R = (J_0R)(I_0R)J = J$. This proves that $J_0R$ is the unique two-sided fractional ideal of $R$ with this property, which we denote $(I_0R)^{-1}$. It follows immediately from uniqueness that $((I_0R)^{-1})^{-1} = I_0R$.
 
 Now for any ideal $I$ of $R$,  define 
$(R : I) = \{x\in D : xI\subseteq R\}$.
  Then we have
  \begin{align*}
  	(R : I_0R)
		&=\{x\in D : xI_0\subseteq R\}
		=\{x\in D : I_0x\subseteq R\}
\end{align*}
because $xI_0R\subseteq R$ if and only if $xI_0\subseteq R$, and $xI_0 = I_0x$ for all $x\in D$ because $I_0$ is contained in the center $Z$.  In particular, $(R : I_0R)$ is a two-sided fractional ideal and it is easy to verify that $(R : I_0R) = (I_0R)^{-1}$.  Indeed, the containments
$$
	R\supseteq (R: I_0R)\cdot I_0R \supseteq (I_0R)^{-1}\cdot (I_0R) = R
$$
show that $(R: I_0R)\cdot I_0R = R$, and similarly $I_0R \cdot (R: I_0R) = R$.
   Therefore, we have
  $$
  	(R : (R : I_0R)) = ((I_0R)^{-1})^{-1} = I_0R.
  $$
  
By \cite[Remark 7.(d)]{kani}, we know that
  $$
  	I(A[I_0R]) \subseteq \bigcap_{Rf\supseteq I_0}Rf
  $$
  where the intersection is taken over all elements $f\in D$.

  A routine verification shows that
 \begin{align*}
 	(R : (R: I_0R)) 
		&= \{x\in D : x\cdot (R: I_0R)  \subseteq R\}\\
		&= \{x\in D : \forall y\in D, \text{ if } I_0y \subseteq R,
			 \ \text{ then } xy \in R \}\\
		&= \{x\in D : \forall y\in D \setminus\{0\}, \text{ if } I_0 \subseteq Ry^{-1},
			 \ \text{ then } x \in Ry^{-1} \}\\
		&=\bigcap_{Ry^{-1} \supseteq I_0R}\{x\in D : x\in Ry^{-1}\}\\
		&=\bigcap_{Ry^{-1} \supseteq I_0R} Ry^{-1}\\
		&=\bigcap_{Rf \supseteq I_0R} Rf
 \end{align*}
where the final equality comes from simply reindexing the intersection with $f = y^{-1}$.

Combining all of the containments above, we see that  
$$
  	I_0R \subseteq I(A[I_0R]) \subseteq \bigcap_{Rf\supseteq I}Rf = (R: (R : I_0R)) = I_0R
  $$ 
  which shows that $I_0R$ is a kernel ideal by definition.
  \end{proof}
  
The lemma above is useful because it shows that the prime ideals appearing in Proposition \ref{inv-key-prop} are actually kernel ideals, which gives us the following important information.  We will write $|H|$ for the rank of a finite subgroup scheme $H$ of $A$, or equivalently, the degree of the isogeny $\pi_H: A \to A/H$.

  \begin{prop} 
  \label{prop:count}
  If $I_0\subseteq Z$ is an invertible ideal, then 
  $$
  	\End_{\FF_q}(A/A[I_0]) = \End_{\FF_q}(A) = R.
$$
 Moreover,
 $$
 	|A[I_0]| = N_{K/\QQ}(I_0)^{2g/[K:\QQ]}.
$$ 
   \end{prop}
    \begin{proof} 
    For convenience, write $B = A/A[I_0]$.
   Because $I_0R$ is a kernel ideal by Lemma \ref{kernel-ideal-lemma}, the endomorphism ring $\End_{\FF_q}(B) $ is equal to the right order of $I_0R$ \cite[Proposition 3.9]{waterhouse}, which we denote by
   $$
   	\OO_r(I_0R) = \{x\in D : (I_0R)\cdot x \subseteq I_0R\}
   $$
     Since $I_0R$ is a two-sided ideal, clearly $R\subseteq \OO_r(I_0R)$.  Conversely, let $x\in \OO_r(I_0R)$.  Then 
    $$
    	Rx = (I_0R)^{-1} (I_0R)x \subseteq (I_0R)^{-1}I_0R =R
$$
 because $I_0R$ is an invertible ideal.  Therefore, $x\in R$ and $\End_{\FF_q}(B) = \OO_r(I_0R) = R$.
 
    To prove the second claim, first assume that $I_0 = \alpha Z$ is a principal ideal.  Then ${A[I_0] = A[\alpha]}$ and $|A[I_0]| = \deg(\alpha)$, so the claim is known \cite[Proposition V.12.12]{milne}.
   
   Now suppose $I_0$ is not principal.  Because $I_0$ is an invertible ideal of $Z$, we can pick an ideal $J_0\subseteq Z$ such that  $I_0J_0 = \lambda Z$ and $N_{K/\QQ}(J_0)$ is coprime to $| A[I_0]|$. Indeed, there are only finitely many prime factors of $|A[I_0]|$, while there are infinitely many prime ideals in the  equivalence class $[I_0]^{-1}\in \Pic(Z)$ by Lemma \ref{lem:choose_inverse}.
     Multiplication of ideals corresponds to composition of isogenies \cite[Proposition 3.12]{waterhouse}, and therefore
   \begin{align*}
   	|A[I_0]| \cdot | B[J_0]|
		&= |A[I_0J_0] |\\
		&=| A[\lambda] |\\
		&=N_{K/\QQ}(\lambda)^{2g/[K:\QQ]}\\
		&=N_{K/\QQ}(I_0)^{2g/[K:\QQ]} N_{K/\QQ}(J_0)^{2g/[K:\QQ]}
   \end{align*}
   Now the fact that the rank of $ A[I_0]$ is coprime to $N_{K/\QQ}(J_0)$ means that $|A[I_0]|$ divides $N_{K/\QQ}(I_0)^{2g/[K:\QQ]}$.  But the same must be true for $J_0$, so $|B[J_0]|$ divides $N_{K/\QQ}(J_0)^{2g/[K:\QQ]}$ as well. Therefore, equality must hold, as claimed.
    \end{proof}

  Because we are ultimately only concerned with separable isogenies, we will restrict our attention to this case now.  Recall that the kernel of a separable isogeny $\phi: A \to A'$ can be identified with a finite subgroup of $A(\overline\FF_{q})$ of cardinality $\deg \phi$.
  
  \begin{lemma} 
If $r\geq 1$, and $\pp\subseteq Z$ is an invertible prime ideal  which corresponds to a separable isogeny,  then
$$
	A[\pp^r] \cong (Z/\pp^r)^{2g/[K: \QQ]}
$$
is an isomorphism of $Z$-modules.
   \end{lemma}
    \begin{proof} 
   First, $A[\pp]$ is a $Z/\pp$-module.  But $Z/\pp$ is a field, so $A[\pp]$ is a vector space, and therefore $A[\pp] \cong (Z/\pp)^m$ for some $m$.  We have $m = 2g/[K:\QQ]$ by counting the cardinality of each side with Proposition \ref{prop:count}.  
   
Now we proceed by induction.  Given $r\geq 2$, we know that $A[\pp^r]$ is a finitely generated module over $Z/\pp^r\cong Z_\pp/\pp^rZ_\pp$.  Because $Z_\pp$ is a discrete valuation ring by Lemma \ref{lem:inv_dvr}, we can apply the structure theorem for finitely generated modules \cite[Theorem 12.1.6]{DF} to deduce that $A[\pp^r]$ is the direct sum of modules of the form $Z_\pp/\pp^iZ_\pp \cong Z/\pp^i$ for $1\leq i\leq r$.

Further, $A[\pp^r]$ contains $A[\pp^{r-1}]$, which is of the form $(Z/\pp^{r-1})^{2g/[K:\QQ]}$ by assumption.  Thus, writing $A[\pp^r] \cong Z/\pp^{r_1}\times \dots\times Z/\pp^{r_s}$ implies that $s = 2g/[K:\QQ]$.  By counting the cardinality, we must have $r_j = r$ for all $1\leq j\leq s$.
    \end{proof}
  
  \subsection{Proof of main result} Now we are ready to prove the main result of this section.
  
   \begin{proof}[Proof of Proposition \ref{inv-key-prop}]   
We factor $(s) = \pp_1^{e_1}\dots \pp_r^{e_r}$.    Notice that for any nonzero $I,J\subseteq R$, we have $A[I]\cap A[J] = A[I + J]$ by definition because $I + J$ is generated by $I \cup J$.  Thus, coprime ideals correspond to subgroups with trivial intersection, and we conclude that we have an isomorphism of $Z$-modules:
$$
	A[s] \cong A[\pp_1^{e_1}]\times \dots \times A[\pp_r^{e_r}].
$$
  For each $1\leq i\leq r$, we see that $A[\pp_i^{e_i}]\cong (Z/\pp_i^{e_i})^{2g/[K:\QQ]}$ by the proposition above.  By the Chinese Remainder Theorem, we conclude that
  	$$
		A[s] \cong (Z/Zs)^{2g/[K:\QQ]}
	$$
	as desired.
  
 Now write $d = 2g/[K:\QQ]$ for convenience. To prove the second claim, we notice that the endomorphism ring of the $Z$-module $A[s]\cong (Z/Zs)^d$ is the ring of $d\times d$ matrices over $Z/Zs$, which we write as $\End_{Z}(A[s]) = \MM_d(Z/Zs)$.   As in the proof of Proposition \ref{gor-key-prop}, we see that $A[s]$ is a faithful $R/Rs$-module, so the map $R/Rs\to \End_{Z}(A[s])$ induced by the natural $R$-module structure on $A[s]$ is injective.  Moreover, $s$ defines a linear map on the lattice $R\subseteq D$, so we have
 $$
 	\#(R/Rs) = N_{D/\QQ}(s) =N_{K/\QQ}(N_{D/K}(s)) = N_{K/\QQ}(s)^{[D:K]},
$$
where $N_{D/\QQ}(s)$ and $N_{D/K}(s)$ denote the determinants of $s: D\to D$ as a linear map over $\QQ$ and $K$, respectively.
On the other hand, it is clear that
$$
	\#\mathrm{M}_d(Z/Zs) = N_{K/\QQ}(s)^{d^2} = N_{K/\QQ}(s)^{[D:K]}
$$
because $d^2 = [D:K]$; see \cite[Theorem 8]{wm}.
  Therefore, $R/Rs$ and $\mathrm{M}_d(Z/Zs)$ have the same cardinality, so the injective ring map $R\to \mathrm{M}_d(Z/Zs)$ is an isomorphism.

 Therefore, to prove that $A[s]$ has exactly one $R$-module structure, it suffices to show that $(Z/Zs)^d$ has exactly one $\MM_d(Z/Zs)$-module structure. Morita equivalence states that every $\MM_d(Z/Zs)$-module $M'$ is isomorphic to $M^d$ for some $Z/Zs$-module $M$, where $M^d$ is given the natural left $\MM_d(Z/Zs)$-module structure defined by applying matrices to column vectors; see \cite[Proposition 1.4]{jacobson}.  Thus we simply need to know that if a $Z$-module $M$ satisfies $M^d \cong (Z/Zs)^d$, then $M \cong Z/Zs$.  But, as above, $s$ is the product of invertible primes, so $M$ must be of the desired form.
 
 Finally, we notice that $\MM_d(Z/Zs)$ is isomorphic to $((Z/Zs)^d)^d$ as a module over itself, which proves the final claim.
     \end{proof}

  \section{Considering the Algebraic Closure}
  \label{sec:Fbar}
Now that we have considered the module structure of the group of rational points of a simple abelian variety over a finite field $\FF_q$, we turn our attention towards the algebraic closure $\Fbar_q$.  Because $\Fbar_q$ is the union of all its finite subfields, we can stitch together the isomorphisms from Propositions \ref{gor-key-prop} and \ref{inv-key-prop} to recover the following theorem.  

As before, given a simple abelian variety $A$ of dimension $g$ over $\FF_q$, we write $R = \End_{\FF_q}(A)$ and define $Z$ to be the center of $R$. Let $[Z : \ZZ]$ denote the rank of $Z$ as a $\ZZ$-module.  Write $S\subseteq Z$ for the set of separable isogenies in $Z$, and $R_S$ (resp. $Z_S$) for the left $R$-submodule (resp. $Z$-submodule) of the endomorphism algebra $R\otimes \QQ$ generated by the set $\{s^{-1} : s\in S\}$.  Equivalently, these can be recognized as localizations by the set $S$.

 \begin{theorem}
 \label{thm:Fbar}
 For  $g\geq 1$, let $A$ be a simple abelian variety over $\FF_q$ of dimension $g$.  Let $R = \End_{\Fbar_q}(A)$, and let $Z$ be the center of $R$. 

 \begin{enumerate}[(a)]
 \item  If $[\QQ(\pi) : \QQ] = 2g$ and $R$ is a Gorenstein ring, then
 $$
 	A(\Fbar_{q}) \cong R_S/R.
 $$
 is an isomorphism of $R$-modules.
 \item If $Z$ is a maximal order, then
 $$
		A(\Fbar_{q}) \cong (Z_S/Z)^{d}.
$$
is an isomorphism of $Z$-modules where $d = 2g/[Z : \ZZ]$.
Moreover, this $Z$-module has exactly one left $R$-module structure, up to isomorphism, and there is an isomorphism
$$
	A(\Fbar_{q})^{ d} \cong R_S/R
$$
as $R$-modules.
 \end{enumerate} 
 \end{theorem}
 
  \begin{proof} 
  Notice  that, in any case, we have
  $$
  	A(\Fbar_q) = \bigcup_{s \in S} A[s] = \bigcup_{n \geq 1} A[\pi^n - 1] = \bigcup_{n\geq 1} A(\FF_{q^n})
$$
where $\pi$ denotes the Frobenius endomorphism of $A$ over $\FF_q$.
Indeed, it is clear that each term contains the next, and the final term equals the first.  This allows us to deduce the theorem after describing only $A[s]$ for $s\in S$.
  
 For part (a), the hypotheses allow us to apply Proposition \ref{gor-key-prop} to obtain isomorphisms $A[s] \cong R/Rs \cong s^{-1}R/R$ for every separable $s\in R$.  In other words, for each $s\in S$, the set $W_s$ of isomorphisms between $A[s]$ and $s^{-1}R/R$ is nonempty.   Moreover, if $s$ and $t$ are two separable endomorphisms such that $s$ divides $t$, then the isomorphism $A[t]\xrightarrow{\sim} t^{-1}R/R$ maps the submodule $A[s]$ isomorphically to $s^{-1}R/R$. Thus the set $\{W_s\}_{s\in S}$ form a projective system of nonempty finite sets, and the projective limit of this system is nonempty \cite[Th\'eor\`eme 1, \S 7.4]{bourbaki}. In particular, there exists a simultaneous choice of isomorphisms $A[s]\to s^{-1}R/R$ for all $s\in S$ that commutes with the natural inclusions of sets, and the result follows by taking the union over all $s\in S$.
 
 Part (b) follows similarly.  Indeed,  for each $s\in S$, Proposition \ref{inv-key-prop} provides an isomorphism  $A[s] \cong (Z/Zs)^d \cong (s^{-1}Z/Z)^d$.  By the same projective limit argument given for part (a), we obtain the desired isomorphism $A(\Fbar_{q}) \cong (Z_S/Z)^{d}$.  Similarly, we obtain the isomorphism $A(\Fbar_{q})^{ d} \cong R_S/R$.
 
 Finally, any two $R$-module structures on $(Z_S/Z)^{d}$ give rise to two $R$-module structures on $(s^{-1}Z/Z)^d$ for each $s\in S$. Since this structure is known to be unique by Proposition \ref{inv-key-prop}, we obtain compatible isomorphisms for all $s\in S$, and yet again obtain the desired isomorphism through the projective limit construction.
  \end{proof}

\end{document}